\documentclass[12pt]{amsart}
\usepackage{amssymb,verbatim,enumerate,ifthen}
\usepackage{cite}
\usepackage[mathscr]{eucal}
\usepackage[utf8]{inputenc}
\usepackage[T1]{fontenc}
\usepackage{esint}
\usepackage{marginnote}
\newtheorem{thm}{Theorem}[section]

\newtheorem{lem}[thm]{Lemma}

\theoremstyle{definition}

\theoremstyle{definition}

\newtheorem{exa}{Example}

\newtheorem*{xrem}{Remark}

\numberwithin{equation}{section}
\def\eq#1{{\rm(\ref{#1})}}
\def\Eq#1#2{\ifthenelse{\equal{#1}{*}}
  {\begin{equation*}\begin{aligned}[]#2\end{aligned}\end{equation*}}
  {\begin{equation}\begin{aligned}[]\label{#1}#2\end{aligned}\end{equation}}}

\def\A{\mathscr{A}}

\def\D{\mathscr{D}}

\def\G{\mathscr{G}}
\def\M{\mathscr{M}}
\def\Nm{\mathscr{N}}
\def\P{\mathscr{P}}
\newcommand\R{\mathbb{R}}
\newcommand\F{\mathbb{F}}
\newcommand\N{\mathbb{N}}
\newcommand\Z{\mathbb{Z}}

\newcommand\Q{\mathbb{Q}}
\def\LP#1{^{\underline{#1}}}
\def\UP#1{^{\overline{#1}}}
\newcommand\vone{{\bf 1}}

 \def\Pr{\mathscr{I}}       

\newcommand{\QA}[1]{\A^{[#1]}}

\newcommand{\abs}[1]{\left| #1 \right| }

\DeclareMathOperator{\supp}{supp}
\title
{From the Ingham--Jessen property to mixed-mean inequalities}

\author{Jacek Chudziak}
\address{Faculty of Mathematics and Natural Sciences, University of Rzesz\'ow, Pigonia 1, 35-310 Rzesz\'ow, Poland}
\email{chudziak@ur.edu.pl}

\author{Zsolt P\'ales}
\address{Institute of Mathematics, University of Debrecen, Egyetem tér 1, 4032 Debrecen, Hungary}
\email{pales@science.unideb.hu}

\author{Pawe{\l} Pasteczka}
\address{Institute of Mathematics, Pedagogical University of Krak\'ow,  Podchor\k{a}\.{z}ych 2, 30-084 Cracow, Poland}
\email{pawel.pasteczka@up.krakow.pl}



\thanks{The research of the second author was supported by the EFOP-3.6.1-16-2016-00022 project. This project is co-financed by the European Union and the European Social Fund.}

\begin{document}
\begin{abstract}
For every symmetric mean $\mathscr{M} \colon \bigcup_{n=1}^\infty I^n \to I$ (where $I$ an interval) and a nonzero function $W \colon \{1,\dots,n\} \to \mathbb{N} \cup \{0\}$, define an $n$-variable mean by
$$\mathscr{M}_W(x):=\mathscr{M}\big(\underbrace{x_1,\dots,x_1}_{W(1)\text{-times}},\dots,\underbrace{x_n,\dots,x_n}_{W(n)\text{-times}}\big) \text{ for }x=(x_1,\dots,x_n) \in I^n.$$
Given two symmetric means $\mathscr{M},\,\mathscr{N} \colon \bigcup_{n=1}^\infty I^n \to I$ satisfying the so-called Ingham--Jessen inequality and some nonzero functions $F_1,\dots,F_k$, $G_1,\dots,G_l \colon \{1,\dots,n\} \to \mathbb{N} \cup \{0\}$, we establish sufficient conditions for  inequalities of the form $$\mathscr{N} \big( \mathscr{M}_{F_1}(x),\dots,\mathscr{M}_{F_k}(x)\big) \le \mathscr{M} \big( \mathscr{N}_{G_1}(x),\dots,\mathscr{N}_{G_l}(x)\big) \qquad(x \in I^n).$$
Our results provide a unified approach to the celebrated inequalities obtained by Kedlaya in 1994 and by Leng--Si--Zhu in 2004 and offer also new families of mixed-mean inequalities. 
\end{abstract}

\maketitle
\section{Introduction}

In this paper a function $\M \colon \bigcup_{n=1}^\infty I^n \to I$ will be called a \emph{mean} if, for all $n\in\N$ 
and $(x_1,\dots,x_n)\in I^n$ 
\Eq{*}{
\min(x_1,\dots,x_n) \le \M(x_1,\dots,x_n) \le \max(x_1,\dots,x_n). 
}

One of the most extensively studied branches of the theory of means is the extension and generalization of the 
important and classical inequalities involving means. Beyond the comparison theory of means, the so-called mixed-mean 
type inequalities are in the focus of many recent research papers. These type of inequalities involve two or more means 
which are substituted into each other.  

For example Kedlaya \cite{Ked94}, giving an affirmative answer to Holland's conjecture \cite{Hol92}, proved that, for 
all $n \in \N$ and $x_1,\dots,x_n>0$, the following inequality holds:
\Eq{*}{
\frac{x_1+\sqrt{x_1x_2}+\cdots+\sqrt[n]{x_1x_2\cdots x_n}}{n} 
\le \sqrt[n]{x_1 \cdot \frac{x_1+x_2}2 \cdots \frac{x_1+x_2+\cdots+x_n}n}.
}
Another result of this type was established by Leng--Si--Zhu \cite{LenSiZhu04} in 2004
\Eq{*}{
&\frac{\sqrt[n-1]{x_2x_3\cdots x_{n}}+\sqrt[n-1]{x_1x_3\cdots x_{n}}+\cdots+\sqrt[n-1]{x_1x_2\cdots x_{n-1}}}{n} \\
&\hspace{25mm}\le \sqrt[n]{\frac{x_2+x_3+\cdots+x_{n}}{n-1} \cdot \frac{x_1+x_3+\cdots+x_{n}}{n-1} \cdots 
\frac{x_1+x_2+\cdots+x_{n-1}}{n-1}}.
}
More general inequalities than this that involve power means (instead the arithmetic and geometric means) were obtained 
by Carlson--Meany--Nelson \cite{CarMeaNel71} in 1971.
For further examples of inequalities of this type we refer the reader to the seminal monograph \cite{Bul03}. The typical tool to prove such inequalities is the Hölder inequality. In our paper we aim to derive the above, as well as some more general inequalities with the help of the so-called Ingham--Jessen inequality. For the setting of the arithmetic and geometric means this states that, for all $n,m \in \N$ and $x_{ij}>0$ ($i\in\N_n$, $j\in\N_m$), 
\Eq{*}{
&\frac{\sqrt[m]{x_{11}x_{12}\cdots x_{1m}}+\sqrt[m]{x_{21}x_{22}\cdots x_{2m}}+\cdots+\sqrt[m]{x_{n1}x_{n2}\cdots 
x_{nm}}}{n} \\
&\hspace{25mm}\le \sqrt[m]{\frac{x_{11}+x_{21}+\cdots+x_{n1}}{n} \cdot \frac{x_{12}+x_{22}+\cdots+x_{n2}}{n} \cdots 
\frac{x_{1m}+x_{2m}+\cdots+x_{nm}}{n}}.
}
Here and also in the sequel, we set $\N_0:=\N \cup \{0\}$ and, for every $k \in \N$, we denote $\N_k:=\{1,\dots,k\}:=[1,k]\cap\N$.
More generally, we say that a pair $(\M,\Nm)$ of means on $I$ \emph{forms an Ingham--Jessen pair} (briefly \emph{is an I--J pair}) if, for $x_{ij}\in I$ ($i\in\N_n$, $j\in\N_m$), the following inequality is satisfied:
\Eq{IJ}{
&\Nm \Big( \M(x_{11},x_{12},\dots,x_{1m}), \M(x_{21},x_{22},\dots,x_{2m}),\dots,\M(x_{n1},x_{n2},\dots,x_{nm}) \Big) \\
&\qquad\le \M \Big( \Nm(x_{11},x_{21},\dots,x_{n1}), \Nm(x_{12},x_{22},\dots,x_{n2}),\dots,\Nm(x_{1m},x_{2m},\dots,x_{nm}) \Big).
}
The Ingham--Jessen pairs were characterized for several families of means. In particular, a pair of power means $(\P_p,\P_q)$ is an  Ingham--Jessen pair if and only if $p \le q$ (see \cite{HarLitPol34}). In a more general setting, a pair of Gini means $(\G_{p,q},\G_{r,s})$ forms an Ingham--Jessen pair if and only if $pqrs=0$ and $\min(p,q) \le \min(r,s) \le \max(p,q) \le \max (r,s)$ (see \cite[Theorem 5]{Pal83a}). Extending these results, P\'ales also characterized the Ingham--Jessen property in the class of quasideviation means \cite{Pal85a}. 

In order to formulate our main statements, we need to introduce some further notations. Given a mean $\M$ on $I$, $n\in\N$, and a non-identically zero function $F \colon \N_n \to \N_0$, we define $\M_F \colon  I^n \to I$ by
\Eq{*}{
\M_F(x_1,\dots,x_n):=\M\big(\underbrace{x_1,\dots,x_1}_{F(1)\text{-times}},\dots,\underbrace{x_n,\dots,x_n}_{F(n)\text{-times}}\big) \qquad\text{for}\quad(x_1,\dots,x_n) \in I^n.
}
For $k,\,l,\,n\in \N$ and sequences $(F_i)_{i=1}^k$ and $(G_j)_{j=1}^l$ of non-identically zero functions mapping $\N_n$ into $\N_0$, and means $\M,\,\Nm$ on $I$, we will investigate the validity of the inequality 
\Eq{E:GIJ}{
 \Nm \big( \M_{F_1}(x),\dots,\M_{F_k}(x)\big) \le \M \big( \Nm_{G_1}(x),\dots,\Nm_{G_l}(x)\big) \qquad \text{ for }x \in I^n.
}
This inequality generalizes the comparison inequality of means $\M$ and $\Nm$. Indeed, for $k=l=1$ and $F_1=G_1\equiv 1$, \eq{E:GIJ} becomes $\M \le \Nm$; while for $k=l=n$ and $F_i=G_i=\textbf{1}_{\{i\}}$ for $i \in \N_n$, \eq{E:GIJ} reduces  to $\Nm \le \M$. Note that the Kedlaya inequality is a particular case of \eq{E:GIJ}, too. In fact, it is enough to take $\M$ and $\Nm$ to be the geometric and arithmetic mean on $(0,+\infty)$, respectively, $k=l=n$, and $F_i=G_i=\textbf{1}_{\{1,\dots,i\}}$ for $i \in \N_n$.

In fact, the Ingham--Jessen inequality \eq{IJ} is a particular case of \eq{E:GIJ}. Indeed, for $k,\,l \in\N$, $n=kl$, 
\Eq{IJ-Fi}{
F_i&=\textbf{1}_{\{(i-1)k+1,(i-1)k+2,\dots,ik\}} &\quad \text{ for }i \in \N_l
}
and
\Eq{IJ-Gj}{
G_j&=\textbf{1}_{\{j,k+j,\dots,(l-1)k+j\}} &\quad \text{ for }j \in \N_k,
}
inequality \eq{E:GIJ} takes the form 
\Eq{*}{
&\Nm \Big( \M(x_1,x_2,\dots,x_k), \M(x_{k+1},x_{k+2},\dots,x_{2k}),\dots,\M(x_{(l-1)k+1},x_{(l-1)k+2},\dots,x_{lk}) \Big) \\
&\qquad\le \M \Big( \Nm(x_1,x_{k+1},\dots,x_{(l-1)k+1}), \Nm(x_2,x_{k+2},\dots,x_{(l-1)k+2}),\dots,\Nm(x_k,x_{2k},\dots,x_{lk}) \Big).
}
The above inequality, with an obvious substitution, becomes equivalent to \eq{IJ}.

The notion of quasiarithmetic means was introduced in 1930s by Aumann, Knopp \cite{Kno28} and Jessen independently and then characterized by Kolmogorov, Nagumo and de Finetti \cite{Kol30,Nag30,Def31}. For a continuous strictly monotone function $ f \colon I \to \R$, we define
the quasiarithmetic mean $\QA{f}$ on $I$ by
\Eq{*}{
\QA{f}(x_1,\dots,x_n):=f^{-1} \Big( \frac{f(x_1)+\cdots+f(x_n)}{n} \Big).
}
Quasiarithmetic means naturally generalize power means. Indeed, whenever $I=\R_+$ and $f=\pi_p$, where $\pi_p(x):=x^p$ if $p\ne 0$ and $\pi_0(x):=\ln x$, then the mean $\QA{f}$ coincides 
with the $p$-th power mean (this is what was noticed by Knopp \cite{Kno28}). These means share most of the properties of 
power means. In particular, it is easy to verify that they are symmetric and repetition invariant. Let us recall here that a mean $\M$ on the interval $I$ is called \emph{symmetric} provided the $n$-variable restriction $M|_{I^n}$ is symmetric for every $n \in \N$; and it is said to be \emph{repetition invariant} if, for all $n,m\in\N$ and $(x_1,\dots,x_n)\in I^n$, the following identity is satisfied
\Eq{*}{
  \M(\underbrace{x_1,\dots,x_1}_{m\text{-times}},\dots,\underbrace{x_n,\dots,x_n}_{m\text{-times}})
   =\M(x_1,\dots,x_n).
}
(cf. \cite{PalPas16}). Quasiarithmetic means share even more properties of power means (cf.\ \cite{Acz48a}, \cite{Kol30}).

Generalizing the above idea, for a given mean $\M$ on $I$ and a homeomorphism $f \colon J \to I$, we define a mean $\M^{[f]}$ on $J$ by 
\Eq{*}{
\M^{[f]}(x_1,\dots,x_n):=f^{-1} \big( \M(f(x_1),\dots,f(x_n)) \big).
}
The mean $\M^{[f]}$ will be called the \emph{$f$-conjugate} of $\M$. In this terminology, quasiarithmetic means are just the conjugates of the arithmetic mean.

It can be observed that if $f$ is continuous and strictly increasing and $(\M,\Nm)$ is an Ingham--Jessen pair, then so is $(\M^{[f]},\Nm^{[f]})$. Additionally, the Ingham--Jessen property of the pairs $(\M,\A)$ and $(\A,\M)$ is equivalent to the convexity and to the concavity of $\M$, respectively. Using these facts, $(\M,\QA{f})$ is an Ingham--Jessen pair if and only if $\M^{[f^{-1}]}$ is convex. Similarly, the Ingham--Jessen property of $(\QA{f},\M)$ is equivalent to concavity of $\M^{[f^{-1}]}$. 
Finally, if $g \colon I \to \R$ is continuous and strictly monotone, then $(\QA{f},\QA{g})$ is an Ingham--Jessen pair if and only if $(\QA{g})^{[f^{-1}]}=\QA{g \circ f^{-1}}$ is concave. Fortunately, the concavity of quasiarithmetic means has been characterized recently in \cite{ChuGlaJarJar19} and in \cite{PalPas18a}. Namely, for a strictly monotone $\mathcal{C}^2$ function $f \colon I \to \R$, the quasiarithmetic mean $\QA{f}$ is concave if and only if either $f''$ is nowhere vanishing and $f'/f''$ is convex and negative, or $f''\equiv 0$.

\section{Conjugated families of probability distributions}

Let $\F$ be a subfield of $\R$.
For $n\in\N$, let $\Pi_n(\F)$ denote the set of all \emph{$\F$-valued probability distributions} ($\F$-p.d.\ for short) over the set $\N_n$, 
i.e., 
\Eq{*}{
  \Pi_n(\F):=\{(\pi_1,\dots,\pi_n)\mid \pi_1,\dots,\pi_n \in[0,1]\cap\F,\,\pi_1+\cdots+\pi_n=1\}.
}
The sequences $(P_1,\dots,P_k)\in\Pi_n(\R)^k$ and $(Q_1,\dots,Q_m)\in\Pi_n(\R)^m$ are said to be \emph{$\F$-conjugated} provided that there exists a matrix $(R_{i,j})\in\Pi_n(\F)^{k\times m}$ such that
\Eq{E:defconj}{
P_i = \frac1m \sum_{j=1}^m R_{i,j}\quad\text{ for all }i\in\N_k \qquad \mbox{and}\qquad
Q_j = \frac1k \sum_{i=1}^k R_{i,j}\quad\text{ for all }j\in\N_m.
}
In such a case, $(R_{i,j})$ is called a \emph{transition matrix} between $(P_1,\dots,P_k)$ and $(Q_1,\dots,Q_m)$. A sequence which is $\F$-conjugated to itself is called \emph{$\F$-selfconjugated}. Obviously, every element of $\Pi_n(\F)$ is $\F$-selfconjugated.

\begin{xrem}
Note that a necessary condition for two sequences $(P_1,\dots,P_k)$ and $(Q_1,\dots,Q_m)$ of $\R$-p.d.\ to be $\F$-conjugated is 
\Eq{*}{
  \frac1k \sum_{i=1}^k P_i=\frac1m \sum_{j=1}^m Q_j.
}
\end{xrem}

In the following examples we present some pairs of $\Q$-conjugated sequences of $\Q$-p.d.

\begin{exa}
For $k, l \in \N$, the sequences $(\frac1k F_i)_{i=1}^l$ and $(\frac1l G_j)_{j=1}^k$, where $F_i$ and $G_j$ are given by 
\eq{IJ-Fi} and \eq{IJ-Gj} are $\Q$-conjugated with the transition matrix $(\vone_{\{(i-1)k+j\}})_{(i,j) \in \N_k \times \N_l}$.
\end{exa}

\begin{exa}[Leng--Si--Zhu \cite{LenSiZhu04}]
\label{exa:LenSiZhu}
For $k,\,n \in \N$ with $\tfrac{n}{2}< k \le n$, the sequence consisting of the normalized characteristic functions of all $k$-elements subsets of $\N_n$ is $\Q$-selfconjugated. In the sequel, we will denote this sequence by $C_n^k$ (ordered lexicographically).
\end{exa}

Further nontrivial examples of the conjugated pairs will be discussed in the last section of this paper.

The following lemma will be useful for our considerations mainly for the case when $\F=\Q$.

\begin{lem}
\label{lem:Qs}
If two sequences of $\F$-p.d.-s are $\R$-conjugated then they are also $\F$-conjugated.
\end{lem}

\begin{proof}
Let $k,\,m,\,n\in\N$, and $(P_1,\dots,P_k)\in\Pi_n(\F)^k$ and $(Q_1,\dots,Q_m)\in\Pi_n(\F)^m$ that are $\R$-conjugated. Then there exists a matrix $(R_{i,j})_{(i,\,j) \in \N_k \times \N_m}$ of real-valued probability distributions such that \eq{E:defconj} is valid. 
Define $r_{i,j,t}:=R_{i,j}(t)$ for $(i,j,t) \in \N_k \times \N_m\times \N_n$.

Now consider the following system of linear equations and inequalities:
\Eq{eqs}{
x_{i,j,t} &\ge 0 &&\text{ for all }(i,j,t) \in \N_k \times \N_m\times \N_n; \\ 
\sum_{t \in \N_n} x_{i,j,t}&=1&& \text{ for all }(i,\,j) \in \N_k \times \N_m;\\
P_i(t)&= \frac1m \sum_{j=1}^m x_{i,j,t} && \text{ for all }(i,t)\in\N_k\times \N_n;\\
Q_j(t) &= \frac1k \sum_{i=1}^k x_{i,j,t} && \text{ for all }(j,t)\in\N_m\times \N_n.
}
By the conjugacy, $(x_{i,j,t}):=(r_{i,j,t})$ is a solution for the above system. On the other hand, the solution set $S$ of this system forms a compact convex polyhedron in the space $\R^{kmn}$. Therefore, $S$ has an extreme point $x^0$ due to the Minkowski(--Krejn--Milmann) Theorem. Then there exist
$kmn$ supporting hyperplanes to $S$ whose intersection is exactly $x^0$. Since the coefficients of these hyperplanes belong to $\F$, therefore the point $x^0$ is a unique solution of a system of linear equations with coefficients belonging to $\F$. Then Cramer's Rule implies that all the coordinates of $x^0$ are in $\F$. Now setting the matrix $R^0_{i,j}(t):=x^0_{i,j,t}$ for $(i,j,t) \in \N_k \times \N_m\times \N_n$, we get that $R^0$ is a transition matrix between the sequences $(P_i)_{i=1}^k$ and $(Q_j)_{j=1}^m$ whose entries are $\F$-p.d.
\end{proof}

Now we are going to formulate our main result. To this end, 
if $W\colon\N_n\to\N_0$ is not identically zero, then its associated $\Q$-valued probability distribution $\D_{W}$ over $\N_n$ is defined as the $n$-tuple
\Eq{*}{
  \D_{W}:=\bigg(\frac{W(1)}{W(1)+\cdots+W(n)},\dots,\frac{W(n)}{W(1)+\cdots+W(n)}\bigg).
}

Our main result is contained in the subsequent theorem.

\begin{thm}
\label{thm:main}
Assume that $(\M,\Nm)$ is an Ingham--Jessen pair of symmetric and repetition invariant means on $I$. Let $k,\,m,\,n\in\N$, $F_i,\,G_j \colon \N_n \to \N_0$ for $i \in \N_k$, $j \in \N_m$. If the sequences $\big(\D_{F_i}\big)_{i=1}^k,\,\big(\D_{G_j}\big)_{j=1}^m$ are $\R$-conjugated, then the inequality \eq{E:GIJ} is valid for every $n$-tuple $x \in I^n$.
\end{thm}

We formulate now an auxiliary result which will play a key role in the proof of this theorem. We need to introduce some further definitions. Let $R$ be a subring of $\R$. We say that $D\subseteq\R$ is an \emph{$R$-interval} if $D$ is of the form $[a,b)$, where $a,b\in R$. The Cartesian product of two $R$-intervals will be called an \emph{$R$-rectangle}. A set being a finite union of $R$-intervals or a finite union of $R$-rectangles is called \emph{$R$-simple}. For an $R$-simple set $H\subseteq\R$, we denote its Lebesgue measure by $|H|$. 
If $D$ and $E$ are $R$-intervals, $\theta\in[0,1]$, then $H\subset D \times E$ is called a \emph{$\theta$-proportional subset of $D\times E$} if 
\begin{enumerate}[(a)] 
 \item for all $x \in D$, $|\{y\in E \colon (x,y)\in H\}|=\theta \cdot |E|$,
 \item for all $y \in E$, $|\{x\in D \colon (x,y)\in H\}|=\theta \cdot |D|$.
\end{enumerate}

The following lemma generalizes Lemma 2.6 in \cite{PalPas18b} to the case $n\ge2$.

\begin{lem}\label{lem:thetaG}
Let $R$ be a subring of $\R$ such that $\Q R\subseteq R$.
Then, for every $R$-rectangle $D\times E$, $n\in\N$ and $\theta_1,\dots,\theta_n \in \Q \cap [0,1]$ with $\theta_1+\cdots+\theta_n\le 1$, there exist $R$-simple disjoint subsets $H_1,\dots,H_n\subseteq D\times E$ such that, for every $k \in \N_n$,  $H_k$ is a $\theta_k$-proportional subset of $D \times E$.
Moreover, if $\theta_1+\cdots+\theta_n=1$, then $(H_k)_{k=1}^n$ is a partition of $D \times E$.
\end{lem}

\begin{proof} Let $D \times E$ be an $R$-rectangle. Then there exists an affine bijection $\varphi \colon [0,1)^2\to D \times E$. Indeed, if $D=[a,b)$ and $E=[c,d)$, then such an affine bijection can be given by
\Eq{*}{
  \varphi(t,s)=((1-t)a+tb,(1-s)c+sd)\qquad\text{ for }\qquad(t,s)\in[0,1)^2.
}
Let $\theta_k \in \Q \cap (0,1]$ for $k \in \N_n$ with $\theta_1+\dots+\theta_n \le 1$. Assume that, for every $k \in \N_n$, $\theta_k=p_k/q$, where $q\in \N$ and $p_k\in\N_q\cup \{0\}$. Put $P_0:=0$, $P_k:=p_1+\cdots+p_k$ for $k \in\N_n$ and
\Eq{*}{
A_{i,j}:=\left[\frac {i-1}q,\frac iq\right) \times \left[\frac {j-1}q,\frac jq\right)\qquad \text{for}\qquad i,\,j\in\N_q.
}
Furthermore, let 
\Eq{*}{
S_k:=\bigcup_{i=1}^{q} \bigcup_{j=i+P_{k-1}}^{i+P_{k}-1} A_{i,j(\text{mod } q)+1}\qquad \text{for}\qquad k\in\N_n.
}
Note that $(S_k)_{k \in \N_n}$ is a family of pairwise disjoint $\Q$-simple subsets of $[0,1)^2$ and, for every $k \in \N_n$, $S_k$ is $\theta_k$-proportional. Therefore, for every $k \in \N_n$, $H_k:=\varphi(S_k)$ is a  $\theta_k$-proportional subset of $D\times E$. Moreover, as $\varphi$ is a bijection and $S_k$-s are pairwise disjoint, so are $H_k$-s. Due to the assumption $\Q R\subseteq R$, it also follows that each $H_k$ is $R$-simple.

If $\theta_1+\cdots+\theta_n=1$, then $P_n=p_1+\cdots+p_n=q$. Hence the family $(S_k)_{k \in \N_n}$ is a partition of $[0,1)^2$, and so the family $(H_k)_{k \in \N_n}$ is a partition of $D \times E$.
\end{proof}

Now we are in the position to proceed with the proof of theorem.

\begin{proof}[Proof of Theorem~\ref{thm:main}]
Let $x=(x_1,\dots,x_n) \in I^n$ and assume that $\big(\D_{F_i}\big)_{i=1}^k$ and $\big(\D_{G_j}\big)_{j=1}^m$ are $\R$-conjugated. In view of Lemma~\ref{lem:Qs}, we have that they are also $\Q$-conjugated, that is, there exists a transition matrix $R=(R_{i,j})_{(i,j) \in \N_k \times \N_m}$ with $\Q$-p.d.\ entries. 
Applying Lemma~\ref{lem:thetaG}, we conclude that, for every $(i,j) \in \N_k \times \N_m$, there exists a family $(H_{i,j}^s)_{s=1}^n$ of pairwise disjoint $\Q$-simple sets such that, for every $(i,j,s)\in \N_k \times \N_m \times \N_n$, the set $H_{i,j}^{s}$ is an $R_{i,j}(s)$-proportional in $[i-1,i)  \times [j-1,j)$. Moreover, this family is a partition of $[i-1,i)  \times [j-1,j)$. Furthermore, since the family is finite, there exists $\ell \in \N$ such that the family $(\ell H_{i,j}^s)_{(i,j,s) \in \N_k \times \N_m \times \N_n}$ is a partition of $[0,\ell k) \times [0,\ell n)$ consisting of $\Z$-simple sets. Thus, the family $(H^s)_{s \in \N_n}$ defined by $H^s:=\bigcup_{(i,j) \in \N_k \times \N_m} \ell H_{i,j}^s$ for $s \in \N_n$ is also a partition of $[0,\ell k) \times [0,\ell n)$ consisting of $\Z$-simple sets.

Define now a matrix $A=(a_{i,j})_{(i,j) \in \N_{\ell k} \times \N_{\ell m}}$ in the following way:
\Eq{*}{
a_{i,j}=s \qquad \text{ whenever }\quad (i-1,j-1) \in H^s, \quad s \in \N_n.
}
For every $b \in \N_{\ell k}$, $c \in \N_{\ell m}$ and $s \in\N_n$, we denote by $B_b(s)$ and $C_c(s)$ the number of appearance of $s$ in the $b$-th row and the $c$-th column of $A$, respectively. 

Fix $b\in \N_{\ell k}$. Then there exists a unique $j_0 \in \N_k$ such that $\tfrac {b-1} \ell \in [j_0-1,j_0)$. Using the $p_{i,j_0}(s)$-proportionality of $H_{i,j_0}^s$ for $(i,s)\in \N_k \times \N_n$, for every $s \in \N_n$, we obtain
\Eq{*}{
B_b(s)&=\abs{H^s \cap \big([0,\ell k) \times \{b\}\big)}
=\abs{\bigcup_{(i,j) \in \N_k \times \N_m} \ell H_{i,j}^s \cap \big([0,\ell k) \times \{b\}\big)}\\
&=\ell \sum_{(i,j) \in \N_k \times \N_m} \abs{H_{i,j}^s \cap \big([0,k) \times \{b/\ell\}\big)}
=\ell \sum_{i \in \N_k} \abs{H_{i,j_0}^s \cap \big([0,k) \times \{b/\ell\}\big)}\\
&
=\ell \sum_{i \in \N_k} \abs{H_{i,j_0}^s \cap \big([0,k) \times [j_0-1,j_0)\big)}=\ell \sum_{i \in \N_k} p_{i,j_0}(s).
}
Thus, in view of \eq{E:defconj}, we get
\Eq{E:Rr}{
B_b(s)=\frac{\ell k}{\sum_{u \in \N_n} G_{j_0}(u)} G_{j_0}(s) \quad \text{ for }s \in\N_n.
}
Similarly, for every $c \in \N_{m\ell}$, we obtain
\Eq{E:Cc}{
C_c(s)=\frac{\ell m}{\sum_{u \in \N_n} F_{i_0}(u)} F_{i_0}(s) \quad \text{ for }s \in\N_n,
}
where $i_0$ is a unique natural number such that $\tfrac {c-1} \ell \in [i_0-1,i_0)$.

After these preparations, we show that the Ingham--Jessen inequality for the pair $(\M,\Nm)$ applied to the matrix $X=(x_{a_{i,j}})_{(i,j) \in \N_k \times \N_m}$ yields inequality \eq{E:GIJ}. To this end, fix $b \in \N_{\ell k}$. Consider $j_0$ as previously and put 
\Eq{*}{
\lambda_{j_0}:=\frac{\ell k}{\sum_{u \in \N_n} G_{j_0}(u)}.
} 
Then, according to the symmetry and repetition invariance of $\Nm$, in view of \eq{E:Rr}, we get
\Eq{*}{
\Nm(x_{b,1},\dots,x_{b,m\ell})&=
\Nm(\underbrace{x_1,\dots,x_1}_{B_b(1)\text{-times}},\dots,\underbrace{x_n,\dots,x_n}_{B_b(n)\text{-times}})
=\Nm(\underbrace{x_1,\dots,x_1}_{\lambda_{j_0} G_{j_0}(1)\text{-times}},\dots,\underbrace{x_n,\dots,x_n}_{\lambda_{j_0} G_{j_0}(n)\text{-times}})\\
&=\Nm(\underbrace{x_1,\dots,x_1}_{G_{j_0}(1)\text{-times}},\dots,\underbrace{x_n,\dots,x_n}_{G_{j_0}(n)\text{-times}})
=\Nm_{G_{j_0}}(x_1,\dots,x_n).
}
Hence, making use of repetition invariance of $\M$, we conclude that
\Eq{*}{
\M\big(\Nm(x_{1,1},\dots,x_{1,m\ell}),\dots,\Nm(x_{k\ell,1},\dots,x_{k\ell,m\ell})\big)
=\M\big(\Nm_{G_1}(x_1,\dots,x_n),\dots,\Nm_{G_k}(x_1,\dots,x_n)\big).
}
Similarly
\Eq{*}{
\Nm\big(\M(x_{1,1},\dots,x_{k\ell,1}),\dots,\M(x_{1,m\ell},\dots,x_{k\ell,m\ell})\big)
=\Nm\big(\M_{F_1}(x_1,\dots,x_n),\dots,\M_{F_m}(x_1,\dots,x_n)\big).
}
Since $(\M,\Nm)$ is an Ingham--Jessen pair, the inequality \eq{E:GIJ} follows.
\end{proof}

\section{Applications}

According to Theorem~\ref{thm:main} each Ingham--Jessen pair of means jointly with a pair of conjugated sequences of function implies validity of a suitable inequality. A few examples of Ingham--Jessen pairs were already mentioned in the introduction, now we will focus our attention on establishing conjugated sequences.
 
\subsection{Conjugacy related to Kedlaya's inequality}
\begin{thm}
For every $n \in \N$, the sequence $( \frac1i\vone_{\N_i} )_{i=1}^n$ is $\Q$-selfconjugated.
\end{thm}

\begin{proof}
Following Kedlaya's idea, define a family $(r_{i,j})_{(i,j) \in \N_n^2}$ of measures on $\N_n$ by 
\Eq{*}{
r_{i,j}(k)
  =\frac{(n-i)!(n-j)!(i-1)!(j-1)!}{(n-1)!(k-1)!(n-i-j+k)!(i-k)!(j-k)!} \qquad \text {for } k \in \N_n.
}
We adopt a convention $m!=\infty$ for negative integers $m$. It has been proved in \cite{Ked94} that this family has the following properties:
\begin{enumerate}[\quad(i)\quad]
\item $r_{i,j}(k) \ge 0$ \quad for $i,j,k \in \N_n$;
\item $\sum_{k=1}^{n}r_{i,j}(k)=1$ \quad for $i,j \in \N_n$;
\item $r_{i,j}(k) = r_{j,i}(k)$ \quad for $i,j,k \in \N_n$;
\item $\sum_{i=1}^{n}r_{i,j}(k)
=\begin{cases} 
n/j & \text{ for }j,k \in \N_n \text{ with } k \le j, \\ 
0 & \text{ for } j,k \in \N_n \text{ with } k>j. 
\end{cases}$
\end{enumerate}
It follows from (i) and (ii) that $(r_{i,j})_{(i,j) \in \N_n^2}$ is a family of probability distributions on $\N_n$. Furthermore, applying (iii) and (iv), one can verify that the matrix $R=(r_{i,j})_{i,\,j \in \N_n}$ validates the $\Q$-selfconfugacy of the family $( \frac1i\vone_{\N_i} )_{i=1}^n$.
\end{proof}

Using that the geometric mean-arithmetic mean, i.e., the pair $(\P_0,\P_1)$ is an Ingham--Jessen pair and then applying Theorem~\ref{thm:main} and the result above, we can derive Kedlaya's inequality.

\subsection{Conjugacy related to combinations}
Let, for every $n \in \N$ and $k \in \N_n$, $C_n^k$ denote the lexicographically ordered sequence consisting of the characteristic functions of all $k$-element subsets of $\N_n$ (cf.\ Example~\ref{exa:LenSiZhu}). 

\begin{thm}
For every $k,\,l,\,n \in \N$ with $\max(k,l)\le n \le k+l-1$ the families $C_n^k$ and $C_n^l$ are conjugated.
\end{thm}

\begin{proof}
Fix $k,l,n \in \N$ with $\max(k,l)\le n \le k+l-1$.  Set 
\Eq{*}{
C_n^k=:\Big(F_1,\dots,F_{\binom n k}\Big)
\qquad\mbox{and}\qquad 
C_n^l=:\Big(G_1,\dots,G_{\binom n l}\Big).
}
We need to show that there exists a transition matrix $R=(r_{ij})\in\Pi_n(\R)^{\binom n k \times \binom n l}$ of probability distributions such that
\Eq{E:FG}{
  \D_F(i) = \frac1{\binom n l} \sum_{j=1}^{\binom n l} R_{i,j}\quad\text{for }i\in\N_{\binom n k} \qquad \mbox{and}\qquad
  \D_G(j) = \frac1{\binom n k} \sum_{i=1}^{\binom n k} R_{i,j}\quad\text{for }j\in\N_{\binom n l}.
}
Note that, by the inequality $n \le k+l-1$, we have that $F_i \cdot G_j$ is not identically zero for all $i,j$. 
Therefore, for every $(i,j) \in \N_{\binom n k} \times \N_{\binom n l}$, we can define $r_{i,j}$ as the uniform probability distribution on the support of $F_i \cdot G_j$.
Then, for $i \in \N_{\binom n k}$ and $s \in \supp F_i$, we have
\Eq{*}{
\sum_{j=1}^{\binom n l} r_{i,j}(s)&=
\sum_{j \colon s \in \supp G_j} r_{i,j}(s)
=\sum_{j \colon s \in \supp G_j} \frac{1}{|\supp F_i \cap \supp G_j|}\\
&=\sum_{m=1}^n \frac{\big|\big\{j \colon s \in \supp G_j,\, |\supp F_i \cap \supp G_j|=m \big\}\big|}{m}\\
&=\sum_{m=1}^n \frac{1}{m}\binom{k-1}{m-1} \binom{n-k}{l-m}
=\sum_{m=1}^n \frac{1}{k}\binom{k}{m} \binom{n-k}{l-m}
=\frac1{k}\binom{n}{l}.
}
Hence, for every $i \in \N_{\binom n k}$, $\binom{n}{l}^{-1}\sum_{j=1}^{\binom n l} r_{i,j}$ is a uniform probability distribution on $\supp F_i$. On the other hand, for every $i \in \N_{\binom n k}$, the function $\frac{F_i}{\sum_{s \in \N_{n}} F_i(s)}$ is also a uniform distribution on $\supp F_i$, which implies the first equality in \eq{E:FG}. The second one can be reached in similar way. This proves the $\Q$-conjugacy of $C_n^k$ and $C_n^l$.
\end{proof}

Since the pair $(\P_0,\P_1)$ is an Ingham--Jessen pair, from Theorem~\ref{thm:main} and the result above, we can obtain the inequality due to Leng--Si--Zhu \cite{LenSiZhu04}.

\subsection{Conjugacy of cyclic blocks}
For the purpose of this section, for a given $k,\,n \in \N$ with $k \le n$, denote
$$O_n^k:=\tfrac1k\big(\vone_{\{1,\dots,k\}},\vone_{\{2,\dots,k+1\}},\dots,\vone_{\{n-k+1,\dots,n\}},\vone_{\{n-k+2,\dots,n,1\}},\dots,\vone_{\{n,1,\dots,k-1\}}\big).$$

We conclude this paper by providing a necessary and sufficient condition for $\Q$-conjugacy of two families $O_n^k$. For the sake of convenience, define the canonical projection $\Pr_n \colon \Z \to \N_n$ by $\Pr_n(k) \equiv k \pmod{n}$ for $k \in \Z$.

\begin{thm}\label{thm:cycle}
Assume that $n,\,k,\,l \in \N$ are such that $k \le l\le n \le k+l-1$. Then the sequences $O_n^k$ and $O_n^l$ are $\Q$-conjugated if and only if there exists a matrix $(A_{i,j})_{(i,j) \in \N_k \times \N_n}$ having nonnegative real entries such that 
\begin{enumerate}[(i)]
\item $\sum_{i=1}^{k} A_{i,j}=1$ for $j \in \N_n$,
\item $\sum_{j=1}^{n} A_{i,j}=n/k$ for $i \in \N_k$,
\item $\sum_{i=1}^{k} A_{i,\Pr_n(i+j-1)}$ is equal to $n/l$ for $j \in \N_l$ and equals zero for $j \in \N_n\setminus\N_l$.
\end{enumerate}
\end{thm}

Before we start the proof, let us apply this result in a simple example.

\begin{exa} 
Each of the sequences $O_7^3$ and $O_7^4$ are $\Q$-conjugated to $O_7^5$. Indeed, one can easily check that the matrices 
\Eq{*}{
\frac{1}{15} \cdot \begin{bmatrix}
15 & 10 & 6 & 3 & 1 & 0 & 0 \\
0 & 5 & 8 & 9 & 8 & 5 & 0 \\
0 & 0 & 1 & 3 & 6 & 10 & 15
\end{bmatrix}
\qquad\text{and}\qquad
\frac{1}{20} \cdot 
\begin{bmatrix}
10 & 11 & 4 & 6 & 4 & 0 & 0 \\
0 & 9 & 11 & 10 & 0 & 5 & 0 \\
0 & 0 & 5 & 0 & 10 & 11 & 9 \\
10 & 0 & 0 & 4 & 6 & 4 & 11
\end{bmatrix}
}
satisfy the conditions (i)--(iii) in the proposition above.
Therefore, in view of Theorem~\ref{thm:main} and the Ingham--Jessen property of the geometric-arithmetic means, the $\Q$-conjugacy of $O_7^3$ and $O_7^5$ implies the following inequality:
\Eq{*}{
&\frac{\sqrt[3]{x_1x_2x_3}+\sqrt[3]{x_2x_3x_4}+\cdots+\sqrt[3]{x_7x_1x_2}}{7} \\
&\hspace{15mm}\le \sqrt[7]{\frac{x_1+x_2+x_3+x_4+x_5}{5} \cdot \frac{x_2+x_3+x_4+x_5+x_6}{5} \cdots 
\frac{x_7+x_1+x_2+x_3+x_4}{5}},
}
Analogously, the conjugancy of $O_7^4$ and $O_7^5$ yields
\Eq{*}{
&\frac{\sqrt[4]{x_1x_2x_3x_4}+\sqrt[4]{x_2x_3x_4x_5}+\cdots+\sqrt[4]{x_7x_1x_2x_3}}{7} \\
&\hspace{15mm}\le \sqrt[7]{\frac{x_1+x_2+x_3+x_4+x_5}{5} \cdot \frac{x_2+x_3+x_4+x_5+x_6}{5} \cdots 
\frac{x_7+x_1+x_2+x_3+x_4}{5}}.
}
\end{exa}

\begin{proof}[Proof of Theorem \ref{thm:cycle}] Assume that the sequences $O_n^k$ and $O_n^l$ are $\Q$-conjugated. Let $R$ be a transition matrix of probability distributions between them.

For the sake of convenience, let $T \colon \N_n \to \N_n$ be given by $T(k):=\Pr_n(k+1)$.
Furthermore, for every matrix $(\mu_{i,j})_{i,\,j\in\N_n}$ with entries being measures on $\N_n$, define 
\Eq{*}{
F((\mu_{i,\,j})_{i,\,j\in\N_n}):=\big(\mu_{T(i),\,T(j)}\circ T\big)_{i,\,j\in\N_n}.
}
Then $F$ is linear, $F^n$ is the identity and, for every $\alpha\in \N_n$, $F^\alpha(R)$ is a transition matrix for the considered pair of sequences. Consequently,
\Eq{*}{
Q:=(q_{i,j})_{i,j\in\N_n}:=\frac{1}{n} \sum_{\alpha=0}^{n-1} F^\alpha(R)
}
is a fixed point of $F$. Moreover, as every convex combination of transition matrices is again a transition matrix, $Q$ is a transition matrix for $O_n^k$ and $O_n^l$. Set
\Eq{*}{
A_{i,j}:=q_{1,\Pr_n(j-l+1)}(i) \qquad \text{for}\quad (i,j) \in \N_k \times \N_n.
}
We are going to prove that the matrix $A=(A_{i,j})$ satisfies the conditions (i)--(iii). Since $q_{1,\Pr_n(j-l+1)}$-s are probability distributions on $\N_k$, we get (i). Furthermore, for $i\in\N_k$, we have 
\Eq{*}{
\frac{1}{n} \sum_{j=1}^{n} A_{i,j}
=\frac{1}{n} \sum_{j=1}^{n} q_{1,\Pr_n(j-l+1)}(i)
=\frac{1}{n} \sum_{j=1}^{n} q_{1,j}(i)
=[O_n^k]_1(i)
=\frac1k\vone_{\{1,\dots,k\}}(i)=\frac1k,
}
which validates (ii). Finally, as $Q$ is a fixed point of $F^{n+1-j}$, for every $i \in \N_n$, we obtain
\Eq{*}{
\sum_{j=1}^{k} A_{j,\Pr_n(i+j-1)}
&=\sum_{j=1}^{k} q_{1,\Pr_n(i+j-l)}(j)
=\sum_{j=1}^{k} q_{T^{n+1-j}(1),T^{n+1-j}\circ \Pr_n(i+j-l)}(T^{n+1-j}(j))\\
&=\sum_{j=1}^{k} q_{\Pr_n(2-j),\Pr_n(i-l+1)}(1)
=q_{1,\Pr_n(i-l+1)}(1)+\sum_{j=n+2-k}^n  q_{j,\Pr_n(i-l+1)}(1)\:.
}
However, since $1 \notin \supp \big([O_n^k]_j\big)$ for $j \in \{2,\dots,n+1-k\}$ and $Q$ is a transition matrix between the sequences $O_n^k$ and $O_n^l$, we have $q_{j,\Pr_n(i-l+1)}(1)=0$ for $j \in \{2,\dots,n+1-k\}$. Thus, using again that $Q$ is a transition matrix, we get
\Eq{*}{
\sum_{j=1}^{k} A_{j,\Pr_n(i+j-1)}
&=q_{1,\Pr_n(i-l+1)}(1)+\sum_{j=n+2-k}^n  q_{j,\Pr_n(i-l+1)}(1)\\
&=\sum_{j=1}^{n} q_{j ,\Pr_n(i-l+1)}(1) 
=n [O_n^l]_{\Pr_n(i-l+1)}(1)
=\begin{cases}
\frac nl \quad  &\text {for } i \in \N_l, \\[1mm]
0  \quad &\text {for } i \in \N_n\setminus\N_l.
 \end{cases}
}
Therefore (iii) holds, and so the proof of the necessity of (i)--(iii) is completed.

To prove the converse implication, fix a matrix $A=(A_{i,j})_{(i,j) \in \N_k \times \N_n}$ with the properties (i)--(iii). Extend the matrix $A$ by putting 
\Eq{A_fill}{
A_{i,j}:=0\quad \text{ for }(i,j) \in (\N_n \setminus \N_k) \times \N_n\,.
}
Let $R=(r_{i,j})_{i,\,j \in \{1,\dots,n\}}$ be a matrix with entries being measures defined as follows
\Eq{*}{
r_{i,j}(s):= A_{\Pr_n(s-i+1),\Pr_n(j-i+1)} \qquad s \in \N_n.
}
It follows from (i) that, for every $(i,j) \in \N_n \times \N_n$, $r_{i,j}$ is a probability distribution on $\N_n$.
Furthermore, for every $j,\,s \in \N_n$, we have
\Eq{*}{
\sum_{i=1}^n r_{i,j}(s)
&= \sum_{i=1}^n A_{\Pr_n(s-i+1),\Pr_n(j-i+1)}
=  \sum_{i=1}^n A_{\Pr_n(s-i+1),\,\Pr_n((j-s+1)+(s-i+1)-1)}
=  \sum_{i=1}^n A_{i,\,\Pr_n((j-s+1)+i-1)}
}
and  so, in view of \eq{A_fill} and (iii), we obtain
\Eq{*}{
\frac1n\sum_{i=1}^n r_{i,j}(s)
= \begin{cases}1/l & \text{ whenever } \Pr_n(s-j+1)\in \N_l, \\0 & \text{ otherwise.} 
\end{cases}\Bigg\}
=[O_n^l]_j(s).
}
Hence the second equality in \eq{E:defconj} is valid with $Q:=O_n^l$. In order to get the first one with $P:=O_n^k$, it is enough to note that, for every $i,s \in \N_n$, it holds
\Eq{*}{
\sum_{j=1}^n r_{i,j}(s)&= \sum_{j=1}^n A_{\Pr_n(s-i+1),\,\Pr_n(j-i+1)} =\sum_{j=1}^n A_{\Pr_n(s-i+1),\,j},
}
whence, by (ii), we get 
\Eq{*}{
\frac1n\sum_{j=1}^n r_{i,j}(s)&= 
\begin{cases}
1/k&\text{whenever }\Pr_n(s-i+1) \in \N_k,\\
0 & \text{otherwise.} \end{cases}\Bigg\}
=[O_n^k]_i(s).
}
Therefore, the matrix $R$ provides transition between the mentioned families.
\end{proof}

\begin{xrem}
Note that in the case where $n \ge k+l$ neither the considered families are conjugated nor the matrix with the properties (i)-(iii) exists. In a light of Theorems~\ref{thm:main} and \ref{thm:cycle} the following question arises naturally. Does for every $k,\,l,\,n\in \N$ with $k \le l \le k+n-1$ there exist a matrix $(A_{i,j})_{(i,j) \in \N_k \times \N_n}$ with the properties (i)--(iii). We mention that for $n\leq 17$, we have verified the existence of such a matrix by manual calculations.
\end{xrem}

However, in the particular case when $l=n-k+1$, we can prove the existence of the matrix $A$ and therefore the transition matrix between $O_n^k$ and $O_n^l$.

\begin{thm} Let $m,n\in\N$ with $m\leq n$. Then the sequences $O_n^m$ and $O_n^{n-m+1}$ are $\Q$-conjugated.
\end{thm}

\begin{proof} Let $k:=\min(m,n-m+1)$ and $l:=\max(m,n-m+1)$. Then, obviously, we have that $k\leq l\leq n=k+l-1$. It is now sufficient to show that $O_n^k$ and $O_n^l$ are $\Q$-conjugated, i.e. construct a matrix $A=(A_{i,j})_{(i,j) \in \N_k \times \N_n}$ which satisfies conditions (i), (ii), and (iii) of Theorem~\ref{thm:cycle}. 

For the costruction of the matrix $A$, we have to recall the notion of Pochhammer symbols and their basic properties. For $x\in\R$ and $\alpha\in\N\cup\{0\}$ set
\Eq{*}{
  x\LP{\alpha}:=\prod_{j=0}^{\alpha-1}(x-j) \qquad\mbox{and}\qquad
  x\UP{\alpha}:=\prod_{j=0}^{\alpha-1}(x+j).
}
They are binded by the following easy identities $x\LP{\alpha}=(x-\alpha+1)\UP{\alpha}$ and $x\UP{\alpha}=(x+\alpha-1)\LP{\alpha}$.
If $x\in \N$ we also have $x\LP{\alpha}=\frac{x!}{(x-\alpha)!}$ ($\alpha \in \{0,\dots,x\}$) and $x\UP{\alpha}=\tfrac{(x+\alpha-1)!}{(x-1)!}$ ($\alpha \in \N$).
Furthermore, the following variants of the binomial theorem are well-known:
\Eq{BT}{
 (x+y)\LP{\alpha}=\sum_{s=0}^\alpha\binom{\alpha}{s}x\LP{s}y\LP{\alpha-s}
  \qquad\mbox{and}\qquad
 (x+y)\UP{\alpha}=\sum_{s=0}^\alpha\binom{\alpha}{s}x\UP{s}y\UP{\alpha-s}.
}

Now, we are in the position to define the matrix $A$: Let
\Eq{*}{
  A_{i,j}:=\binom{k-1}{i-1} \frac{(j-1)\LP{i-1}(n-j)\LP{k-i}}{(n-1)\LP{k-1}}.
}
Trivially, all the entries of $A$ are nonnegative.
To show that condition (i) of Theorem~\ref{thm:cycle}, we use the first formula from \eq{BT}. Then, for $j\in\N_n$, we have
\Eq{*}{
\sum_{i=1}^{k}A_{i,j}
=\sum_{i=1}^k \binom{k-1}{i-1} \frac{(j-1)\LP{i-1}(n-j)\LP{k-i}}{(n-1)\LP{k-1}} 
=\frac{(n-1)\LP{k-1}}{(n-1)\LP{k-1}}=1.
}
For $i\in\N_k$, we have
\Eq{*}{
\sum_{j=1}^{n}A_{i,j}
&=\sum_{j=1}^{n} \binom{k-1}{i-1} \frac{(j-1)\LP{i-1}(n-j)\LP{k-i}}{(n-1)\LP{k-1}}
=\sum_{j=i}^{n-k+i} \binom{k-1}{i-1} \frac{(j-1)\LP{i-1}(n-j)\LP{k-i}}{(n-1)\LP{k-1}}\\ 
&=\sum_{\alpha=0}^{n-k} \binom{k-1}{i-1} \frac{(\alpha+i-1)\LP{i-1}(n-\alpha-i)\LP{k-i}}{(n-1)\LP{k-1}}\\ 
&=\sum_{\alpha=0}^{n-k}\frac{(k-1)!}{(k-i)!(i-1)!} \cdot \frac{(n-k)!(\alpha+i-1)!(n-i-\alpha)!}{(n-1)!\alpha!(n-k-\alpha)!} \\
&=\frac{(k-1)!}{(n-1)!}\sum_{\alpha=0}^{n-k} \frac{(n-k)!}{\alpha!(n-k-\alpha)!}\cdot\frac{(\alpha+i-1)!(n-i-\alpha)!}{(i-1)!(k-i)!} \\
&=\frac{(k-1)!}{(n-1)!}\sum_{\alpha=0}^{n-k} \binom{n-k}\alpha i\UP{\alpha} (k-i+1)\UP{n-k-\alpha} \\
&=\frac{(k-1)!}{(n-1)!}(k+1)\UP{n-k} =\frac{k! (k+1)\UP{n-k}}{k (n-1)!}=\frac{n!}{k (n-1)!}=\frac nk.
}

Finally, observe that if $i\in\N_k$ and $j\in\N_{n-k+1}$, then, we have
$i+j-1\leq n$, and hence $\Pr_n(i+j-1)=i+j-1$. For an arbitrary $j\in\N_{n-k+1}$ we calculate
\Eq{*}{
\sum_{i=1}^{k} A_{i,\Pr_n(i+j-1)}
&=\sum_{i=0}^{k-1} A_{i+1,i+j}
=\sum_{i=0}^{k-1} \binom{k-1}{i} \frac{(i+j-1)\LP{i}(n-i-j)\LP{k-i-1}}{(n-1)\LP{k-1}}\\
&=\sum_{i=0}^{k-1} \binom{k-1}{i} \frac{j\UP{i}(n-j-k+2)\UP{k-i-1}}{(n-k+1)\UP{k-1}}
=\frac{(n-k+2)\UP{k-1}}{(n-k+1)\UP{k-1}}=\frac{n}{n-k+1}. 
}

To conclude the proof observe that, as all entries of $A$ are nonnegative and
\Eq{*}{
\sum_{ j=n-k+2}^n\sum_{i=1}^{k} A_{i,\Pr_n(i+j-1)}=\sum_{j=1}^n \sum_{i=1}^k  A_{i,j} - \sum_{j=1}^{n-k+1}
\sum_{i=1}^{k} A_{i,\Pr_n(i+j-1)}=n-\frac{(n-k+1)n}{n-k+1}=0,
}
we obtain $A_{i,\Pr_n(i+j-1)}=0$ for all $i \in \N_k$ and $j \in \N_n \setminus \N_{n-k+1}$. In fact it is a particular case of a more general rule which emerges directly from (iii) restricted to $j \in \N_l$ and (i).
\end{proof}

This theorem in the classical arithmetic-geometric setting yields the inequality
\Eq{*}{
&\frac{\sqrt[n-m+1]{x_1x_2\cdots x_{n-m+1}}+\sqrt[n-m+1]{x_2x_3\cdots x_{n-m+1}}+\cdots+\sqrt[n-m+1]{x_nx_1\cdots x_{n-m}}}n\\\
&\qquad \le \sqrt[n]{\frac{x_1+\cdots+x_m}m \cdot \frac{x_2+\cdots+x_{m+1}}m\cdots\frac{x_n+x_1+\cdots+x_{m-1}}m},
}
which is valid for $n \ge 2$, $m \le n$ and $x=(x_1,\dots,x_n) \in \R_+^n$. Remarkably, in the boundary case 
$m=n$ it reduces to the comparability of arithmetic and geometric mean. For $n=m-1$ it is just the Leng--Si--Zhu inequality.

Obviously arithmetic and geometic means in the inequality above can be replaced by power means or (even more generally) by any Ingham-Jessen pair of means.


\end{document}